\documentclass[12pt]{amsart}
\usepackage{amsmath}
\usepackage{amsfonts}
\usepackage{amssymb}
\usepackage{amscd}
\usepackage{mathtools}
\usepackage{amsxtra}
\usepackage{amsthm}
\usepackage{graphicx}
\usepackage[abbrev,alphabetic]{amsrefs}
\RequirePackage[dvipsnames,usenames]{color}
\usepackage{soul,xcolor}
\setstcolor{red}
\usepackage{stmaryrd}
\usepackage{booktabs}
\usepackage{multirow}
\newtagform{tiny}{\tiny(}{)}
\usepackage{threeparttable}

\usepackage{mathtools}
\usepackage{hyperref}
\usepackage[margin=1.25in]{geometry}
\usepackage{mathrsfs}

\usepackage{amsthm}
\usepackage{comment}
\usepackage[all,cmtip]{xy}
\usepackage{tikz-cd}
\usetikzlibrary{cd}

\tikzcdset{
  cells={font=\everymath\expandafter{\the\everymath\displaystyle}},
}

\usepackage[all]{xy}

\usepackage{cleveref}

\makeatletter
\def\@tocline#1#2#3#4#5#6#7{\relax
  \ifnum #1>\c@tocdepth 
  \else
    \par \addpenalty\@secpenalty\addvspace{#2}%
    \begingroup \hyphenpenalty\@M
    \@ifempty{#4}{%
      \@tempdima\csname r@tocindent\number#1\endcsname\relax
    }{%
      \@tempdima#4\relax
    }%
    \parindent\z@ \leftskip#3\relax \advance\leftskip\@tempdima\relax
    \rightskip\@pnumwidth plus4em \parfillskip-\@pnumwidth
    #5\leavevmode\hskip-\@tempdima
      \ifcase #1
       \or\or \hskip 1em \or \hskip 2em \else \hskip 3em \fi%
      #6\nobreak\relax
    \hfill\hbox to\@pnumwidth{\@tocpagenum{#7}}\par
    \nobreak
    \endgroup
  \fi}
\makeatother

\renewcommand{\P}{\mathbb{P}}
\newcommand{\Z}{\mathbb{Z}}

\newcommand{\cO}{\mathcal{O}}

\newcommand{\m}{\mathfrak{m}}

\newcommand{\fp}{\mathfrak{p}}

\newcommand{\Proj}{\mathrm{Proj}}
\newcommand{\wt}{\widetilde}

\def\var{\overline}

\DeclareMathOperator{\Spec}{Spec}

\newcommand{\bs}{\boldsymbol{s}}

\theoremstyle{plain}
\newtheorem{theorem}{Theorem}[section]

\newtheorem{proposition}[theorem]{Proposition}

\newtheorem{lemma}[theorem]{Lemma}

\newtheorem{corollary}[theorem]{Corollary}

\newtheorem{conjecture}[theorem]{Conjecture}

\newtheorem{claim}[theorem]{Claim}
\newtheorem*{claim*}{Claim}

\newtheorem{theoremA}{Theorem}

\theoremstyle{definition}

\newtheorem{notation}[theorem]{Notation}

\newtheorem*{setup*}{Setup}

\theoremstyle{remark}

\newtheorem*{ackn}{Acknowledgements}

\theoremstyle{plain}

\numberwithin{equation}{section}

\crefname{theorem}{Theorem}{Theorems}
\crefname{proposition}{Proposition}{Propositions}
\crefname{lemma}{Lemma}{Lemmas}
\crefname{corollary}{Corollary}{Corollaries}
\crefname{conjecture}{Conjecture}{Conjectures}
\crefname{claim}{Claim}{Claims}
\crefname{notation}{Notation}{Notations}
\crefname{remark}{Remark}{Remarks}
\crefname{example}{Example}{Examples}
\crefname{definition}{Definition}{Definitions}
\crefname{theoremA}{Theorem}{Theorems}

\title{Perfectoid splitting and global $+$-regularity for smooth hypersurfaces}

\author{Shou Yoshikawa}
\address{Institute of Science Tokyo, Tokyo 152-8551, Japan}
\email{yoshikawa.s.9fe9@m.isct.ac.jp}

\begin{document}

\begin{abstract}
In this paper, we prove that smooth Calabi--Yau hypersurfaces of degree $d$ over complete unramified discrete valuation rings with residue characteristic $p$ are perfectoid split if $p$ is larger than the relative dimension and $p\nmid d$. We also show that unramified lifts of smooth Fano hypersurfaces $X$ over fields of characteristic $p>0$ are globally $+$-regular if $p\ge \dim X$ and $p\nmid d$.
\end{abstract}

\maketitle

\section{Introduction}

Let $S$ be a scheme that is smooth and dominant over $\Spec \mathbb{Z}$, and let
$Y_S \to S$ be a smooth proper morphism whose fibers are Calabi--Yau varieties.
An important open problem, both in arithmetic geometry and in the algebraic geometry of positive characteristic, is the following weak ordinarity conjecture.

\begin{conjecture}\label{dense-ordinary-conj}
The fiber $Y_{\mathfrak p}$  is weakly ordinary for a Zariski-dense set of closed points
$\mathfrak p\in S$.
\end{conjecture}

Conjecture~\ref{dense-ordinary-conj} is known in low dimensions. In dimension $1$, namely
for elliptic curves, it follows from the classical theory of reduction of elliptic curves;
moreover, for elliptic curves over $\mathbb{Q}$, Elkies proved that there exist infinitely
many supersingular primes, so there are also infinitely many primes of non-ordinary
reduction \cite{Elkies1987}. In dimension $2$, Bogomolov and Zarhin proved that every K3
surface over a number field acquires ordinary reduction at all but a density-zero set of
primes after a finite extension of the base field \cite{Bogomolov-Zarhin}. Thus the
conjecture holds for elliptic curves and for K3 surfaces, but remains widely open in higher
dimensions, even for hypersurfaces.
For Calabi--Yau varieties, weak ordinarity is known to be equivalent to global $F$-splitting.
Hence Conjecture~\ref{dense-ordinary-conj} may also be viewed as a positive-characteristic
$F$-splitting statement.

The first main result of this paper is a mixed-characteristic analogue of
Conjecture~\ref{dense-ordinary-conj} for hypersurfaces.
The theory of perfectoid spaces, introduced by Scholze \cite{Scholze-perfectoid},
provides a framework for studying mixed-characteristic phenomena via comparison
with characteristic $p$. 
A fundamental feature of this theory is the tilting operation,
which relates objects in mixed characteristic to corresponding objects in characteristic
$p$. In this way, perfectoid geometry makes it possible to formulate and study
mixed-characteristic analogues of phenomena that are governed in characteristic $p$ by the
Frobenius morphism.

A mixed-characteristic analogue of $F$-splitting is \emph{perfectoid splitting},
introduced in \cite{IY26}. 
More precisely, a quasi-compact separated scheme $X$ is said to be
\emph{perfectoid split} if there exists an affine morphism $\pi \colon Y \to X$ such that
the $p$-adic formal completion $\widehat{Y}$ is a perfectoid formal scheme and the natural
map $\mathcal O_X \to \pi_*\mathcal O_Y$ is ind-split, that is, a filtered colimit of split
morphisms in $\mathrm{QCoh}(X)$. Thus perfectoid splitting is naturally viewed as a
mixed-characteristic analogue of $F$-splitting.

Moreover, perfectoid splitting is useful beyond the analogy with positive characteristic:
Bhatt proved that perfectoid split regular schemes satisfy a Kodaira-type vanishing theorem
\cite{Bhatt}*{Proposition~11.2.16}. As in the case of $F$-splitting in positive
characteristic (cf.~\cite{MR85}), this shows that perfectoid splitting is a useful notion
for proving vanishing theorems in mixed characteristic.

In this language, the following conjecture may be regarded as a mixed-characteristic
counterpart of Conjecture~\ref{dense-ordinary-conj}.

\begin{conjecture}\label{mixed-analog}
The localization
\[
Y_{S,\mathfrak p}:=Y_S\times_S \Spec \mathcal{O}_{S,\mathfrak p}
\]
is perfectoid split for a Zariski-dense set of closed points $\mathfrak p \in S$.
\end{conjecture}


Surprisingly, for hypersurfaces we obtain a much stronger conclusion: if $Y_S$ is a hypersurface in a projective space over $S$, then the localization
\[
Y_{S,\mathfrak p}:=Y_S\times_S \Spec \mathcal{O}_{S,\mathfrak p}
\]
is perfectoid split for almost all $\mathfrak p$ (cf.~Corollary~\ref{weak-ordinary}) by the following result.

\begin{theoremA}[Theorem~\ref{CY}]\label{main-CY}
Let $X$ be a smooth Calabi--Yau hypersurface of degree $d$ in a projective space over a complete discrete valuation ring $C$ with maximal ideal $(p)$.
If $p$ is larger than the relative dimension of $X$ and $p \nmid d$, then $X$ is perfectoid split.
\end{theoremA}

\noindent
Thus Theorem~\ref{main-CY} indicates that perfectoid splitting should be viewed as more
than a formal mixed-characteristic analogue of Frobenius splitting. At least for
hypersurfaces, it suggests the possibility that perfectoid splitting applies to a wider
class of varieties than Frobenius splitting.

To prove Theorem~\ref{main-CY}, we first reduce the global statement to a ring-theoretic one.
By the local-global correspondence between perfectoid splitting and perfectoid purity developed in \cite{IY26}, it is enough to study the perfectoid purity of the homogeneous coordinate ring of the given hypersurface. 
Here perfectoid purity is a mixed-characteristic analogue of $F$-purity,
introduced in \cite{p-pure}. For this ring-theoretic problem, the main tool is the
splitting-order sequence introduced in \cite{Yoshikawa25-cri}, together with the criterion
for perfectoid purity in terms of this sequence established there.
For a homogeneous hypersurface
\[
X=\Proj\  C[x_0,\dots,x_N]/(f),
\]
we establish a uniform bound on the splitting-order sequence of a lift of $f$ under a simple
numerical condition depending only on
\[
(p,N,\deg f,\sum_{i=0}^N \deg(x_i)).
\]
In the Calabi--Yau case this specializes to the inequality $p>N-1$.

The bound in Theorem~\ref{main-CY} is already sharp in dimension $2$. Indeed, in
Section~3 we study weighted hypersurface K3 surfaces and show that if $p>2$, then every
unramified lift as a hypersurface is perfectoid split. 
On the other hand, when $p=2$, there exist
supersingular hypersurface K3 surfaces admitting lifts that are not perfectoid split; see
Theorem~\ref{p-purity-K3} and \cite{TY26-K3}. 
Moreover, the characteristic $2$ case for K3 surfaces exhibits an additional phenomenon:
for every supersingular hypersurface K3 surface, one can choose a lift that is perfectoid split.

We next turn to Fano hypersurfaces. Here the relevant mixed-characteristic analogue of
global $F$-regularity is global $+$-regularity, developed in recent work on
mixed-characteristic algebraic geometry and the minimal model program
\cites{BMPSTWW23,TY23}.

Let $X$ be a smooth hypersurface of degree $d$ in a well-formed weighted projective
space $\mathbb{P}_k(q_0,\ldots,q_N)$ over a field $k$ of characteristic $p>0$, and set
$Q:=q_0+\cdots+q_N$. It is known that if
\[
p>\frac{(N+1)d-Q-2d}{Q-d},
\]
then $X$ is globally $F$-regular. The second main result of this paper is a
mixed-characteristic analogue of this boundedness statement.

\begin{theoremA}[Theorem~\ref{Fano-smooth}]\label{main-degree-cond}
In the above setting, if $p \nmid d$ and
\[
(Q-d)p^2+(p-N)d+Q>0,
\]
then every unramified lift of $X$ is globally $+$-regular.
In particular, if $p\ge \dim X$ and $p \nmid d$, then every unramified lift of $X$ is globally $+$-regular.
\end{theoremA}

\begin{ackn}
The author wishes to express his gratitude to Teppei Takamatsu and Jakub Witaszek for valuable discussions.
The author was supported by JSPS KAKENHI Grant number JP24K16889.
\end{ackn}

\section{Numerical criteria from splitting-order sequences}
In this section, we establish uniform bounds for the splitting-order sequence of a weighted homogeneous hypersurface. 
These bounds will imply perfectoid purity in the Calabi--Yau case and eventual vanishing in the Fano case.

We fix a prime number $p$.

\begin{notation}
Throughout this section, we use the following notation.

\begin{itemize}
\item
$C$ denotes a complete discrete valuation ring of mixed characteristic $(0,p)$
with maximal ideal $(p)$ and residue field $k$.
The Frobenius lift of $C$ is denoted by $\phi_C$.

\item
\[
A:=C[x_0,\dots,x_N]
\]
is endowed with the weighted grading
\[
\deg(x_i)=q_i>0 \qquad (0\le i\le N),
\]
and we set
\[
Q:=q_0+\cdots+q_N.
\]

\item
$f\in A$ is a homogeneous element of degree $d$.

\item
$\phi\colon A\to A$ denotes the Frobenius lift determined by
\[
\phi(x_i)=x_i^p \qquad (0\le i\le N), \quad \phi|_{C}=\phi_C
\]
and
\[
\Delta(f):=\frac{f^p-\phi(f)}{p}.
\]
We note that $\deg(\Delta(f))=pd$.

\item
We write
\[
\m:=(p,x_0,\dots,x_N),
\qquad
\m^{[p^e]}:=(p,x_0^{p^e},\dots,x_N^{p^e})
\]
for every $e\ge 1$.

\item
We set
\[
\bar A:=A/pA,
\qquad
\bar f:=f\bmod p,
\qquad
\bar{\m}:=(x_0,\dots,x_N)\subset \bar A,
\]
and
\[
X:=\Proj\bigl(\bar A/(\bar f)\bigr).
\]

\item
For $0\le i\le N$, we write
\[
f_i:=\frac{\partial f}{\partial x_i},
\qquad
J(f):=(f_0,\dots,f_N).
\]

\item
Following \cite{Yoshikawa25-cri}, we denote by
\[
\bs(f)=(s_0,s_1,s_2,\dots)
\]
the splitting-order sequence of $f$. If $s_{n-1} \leq p-1$, it is characterized by
\[
s_n=\max\{\,0\le s\le p\mid f(s_1,\dots,s_{n-1},s)\in \m^{[p^n]}\,\},
\]
and
\[
f(s_1,\dots,s_{n-1},s)
=
f(s_1,\dots,s_{n-2},s_{n-1}+1)^p\,\Delta(f)^{s_{n-1}}f^{p-s}
\]
for every $n\ge 1$.

\item
Since $X$ is a hypersurface in the weighted projective space $\mathbb P(q_0,\dots,q_N)$,
we have
\[
\dim X=N-1.
\]

\item
For $m\ge 2$, we consider the numerical condition
\[
(N_m)\qquad d(p^m-p-1)<p^mQ+Q-(N+1)d.
\]
\end{itemize}
\end{notation}

The following lemma is proved by the same argument as in \cite{Bhatt-Singh}*{Lemma~3.2}, but since we work in a weighted setting, we include the proof for the convenience of the reader.

\begin{lemma}\label{lem:weighted-jacobian-containment}
Assume that $J(\bar f)\subset \bar A$ is $\bar\m$-primary. Then
\[
\bar A_{\ge (N+1)d-2Q+1}\subseteq J(\bar f).
\]
\end{lemma}

\begin{proof}
Since $J(\bar f)=(\bar f_0,\dots,\bar f_N)$ is $\bar\m$-primary, the sequence $\bar f_0,\dots,\bar f_N$ is a homogeneous regular sequence in $\bar A$, with
\[
\deg(\bar f_i)=d-q_i.
\]
Set $R:=\bar A/J(\bar f)$.
Because $\omega_{\bar A}\cong \bar A(-Q)$,  we have
\[
\omega_R\cong R\Bigl(\sum_{i=0}^N(d-q_i)-Q\Bigr)
=R\bigl((N+1)d-2Q\bigr).
\]
Since $R$ is Artinian Gorenstein, its socle degree is $(N+1)d-2Q$. Therefore
\[
R_m=0
\qquad\text{for every }m>(N+1)d-2Q,
\]
which is equivalent to
\[
\bar A_{\ge (N+1)d-2Q+1}\subseteq J(\bar f).
\qedhere
\]
\end{proof}

\begin{lemma}\label{lem:weighted-colon-estimate}
For every integer $k\ge 0$ and every $e\ge 1$, one has
\[
\bigl(\m^{[p^e]}\var{A}:\bar A_{\ge k}\bigr)
\subseteq
\m^{[p^e]}\var{A}+\bar A_{\ge p^eQ-Q-k+1}.
\]
\end{lemma}

\begin{proof}
Let $G\in (\m^{[p^e]}\var{A}:\bar A_{\ge k})$ be a homogeneous element.
We assume $G\notin \m^{[p^e]}\var{A}$.
Let $M=x_0^{a_0}\cdots x_N^{a_N}$ be a monomial such that the coefficient of $M$ in $G$ is non-zero.
Then we have $a_i \le p^e-1$ for every $0 \leq i \leq N$.
Suppose that
\[
\deg(G)<p^eQ-Q-k+1,
\]
then $b_i=p^e-1-a_i \geq 0$ satisfy
\[
\sum_{i=0}^N b_iq_i=\sum_{i=0}^N(p^e-1-a_i)q_i=(p^e-1)Q-\deg(G) \geq k.
\]
Then we have
\[
G x_0^{b_0}\cdots x_{N}^{b_N} \in G \var{A}_{\geq k} \subseteq \m^{[p^e]}\var{A}.
\]
This contradicts the fact that
\[
Mx^{b_0}\cdots x_{N}^{b_N}=x_0^{p^e-1}\cdots x_N^{p^e-1} \notin \m^{[p^e]}\var{A}.
\]

Therefore we have
\[
(\m^{[p^e]}\var{A}:\bar A_{\ge k})\subseteq \m^{[p^e]}\var{A}+\bar A_{\ge p^eQ-Q-k+1}.
\qedhere
\]
\end{proof}

\begin{lemma}\label{lem:jacobian-degree-membership}
Let $H\in A$ be a homogeneous element. Assume that
\[
HJ(f)\subseteq \m^{[p^e]}
\qquad\text{and}\qquad
\deg(H)<p^eQ+Q-(N+1)d.
\]
Then
\[
H\in \m^{[p^e]}.
\]
\end{lemma}

\begin{proof}
Let $\bar H$ denote the image of $H$ in $\bar A=A/pA$.
Then
\[
\bar H\in \bigl(\m^{[p^e]}\var{A}:J(\bar f)\bigr).
\]
By Lemma~\ref{lem:weighted-jacobian-containment} and Lemma~\ref{lem:weighted-colon-estimate}, the right-hand side is
\[
\bigl(\m^{[p^e]}\var{A}:J(\bar f)\bigr) \subseteq \bigl(\m^{[p^e]}\var{A}:A_{\geq (N+1)d-2Q+1}\bigr) \subseteq \m^{[p^e]}\var{A}+\var{A}_{\geq (p^e+1)Q-(N+1)d}.
\]
Since $\deg(H)<(p^e+1)Q-(N+1)d$, we have $H \in \m^{[p^e]}$, as desired.
\end{proof}

\begin{lemma}\label{lem:delta-derivative}
For every $0\le i\le N$, one has
\[
\frac{\partial \Delta(f)}{\partial x_i}
=
f^{p-1}f_i-x_i^{p-1}\phi(f_i).
\]
In particular, modulo $p$,
\[
\frac{\partial \Delta(f)}{\partial x_i}
\equiv
f^{p-1}f_i-x_i^{p-1}f_i^p
\pmod p.
\]
\end{lemma}

\begin{proof}
We have
\[
\frac{\partial \Delta(f)}{\partial x_i}
=
\frac{p f^{p-1}f_i-\frac{\partial \phi(f)}{\partial x_i}}{p}=f^{p-1}f_i-x_i^{p-1}\phi(f_i).
\]

Finally, since $\phi$ is a Frobenius lift, we have 
\[
\frac{\partial \Delta(f)}{\partial x_i}
\equiv
f^{p-1}f_i-x_i^{p-1}f_i^p
\pmod p.
\qedhere
\]
\end{proof}

\begin{theorem}\label{thm:sn-le-p-minus-1}
Assume that the Jacobian ideal $J(\bar f)\subset \bar A$ is $\bar{\m}$-primary.

\begin{enumerate}
\item
Let $n\ge 2$. If $(N_m)$ holds for every $m$ with $2\le m\le n$, then
\[
s_i\le p-1
\qquad\text{for every }1\le i\le n.
\]

\item
If $Q\ge d$ and $(N_2)$ holds, then
\[
s_n\le p-1
\qquad\text{for every }n\ge 1.
\]

\item
In particular, if $Q=d$ and $p>\dim X$, then
\[
s_n\le p-1
\qquad\text{for every }n\ge 1.
\]
\end{enumerate}
\end{theorem}

\begin{proof}
We set
\[
L_n:=f(s_1,\dots,s_{n-1},s_n+1)=L_{n-1}^p\Delta(f)^{s_{n-1}}f^{p-(s_n+1)}
\]
for every $n\ge 1$ with $s_n\le p-1$.
Thus, the degree is
\[
\deg(L_n)=p\deg(L_{n-1})+(ps_{n-1}+p-(s_n+1))d
\]
and $\deg(L_1)=(p-s_1-1)d$.
Therefore, we obtain
\begin{equation}\label{eq:deg-L}
   \deg(L_n)=(p^n-s_n-1)d 
\end{equation}
for every $n\ge 1$ with $s_n\le p-1$.
We prove simultaneously the following two assertions by induction on $n$:

\noindent
$(A_n)$ One has
\[
s_n\le p-1.
\]

\noindent
$(B_n)$ If $1\le s_n\le p-1$, then
\[
L_n(f,J(f))\subseteq \m^{[p^n]}.
\]

We begin with $n=1$. Since $s_1=p$ would imply $1=f^{p-p}\in \m^{[p]}$, we get $s_1\le p-1$, proving $(A_1)$.

Next, assume $1\le s_1\le p-1$. By definition of $s_1$,
\[
L_1f=f^{p-s_1}\in \m^{[p]}.
\]
Differentiating, we obtain
\[
(p-s_1)L_1f_i\in \m^{[p]}
\qquad (0\le i\le N).
\]
Since $1\le s_1\le p-1$, the coefficient $p-s_1$ is a unit modulo $p$, and hence
\[
L_1f_i\in \m^{[p]}
\qquad (0\le i\le N).
\]
Therefore $L_1(f,J(f))\subseteq \m^{[p]}$, so $(B_1)$ also holds.

Now fix $n\ge 2$, and assume that $(A_i)$ and $(B_i)$ hold for all $1\le i\le n-1$.
For $0\le k,t\le p-1$, set
\[
H_{k,t}:=L_{n-1}^pf^k\Delta(f)^t.
\]
We note that if $t \leq s_{n-1}-1$, we have
\begin{align*}
   \deg(H_{k,t})
   &=p\deg(L_{n-1})+(k+pt)d \leq p(p^{n-1}-s_{n-1}-1)d+(p-1+p(s_{n-1}-1))d \\
   &=(p^{n}-p-1)d<p^nQ+Q-(N+1)d,
\end{align*}
where the first equality follows from \eqref{eq:deg-L} and the last equality follows from the condition $(N_n)$.
Thus, we have
\begin{equation}\label{eq:degree-H}
    \deg(H_{k,t}) <p^nQ+Q-(N+1)d \quad \text{for $t \leq s_{n-1}-1$}.
\end{equation}

\noindent
{\bf Proof of $(B_n)$.}
First, we note that
\[
L_{n-1}f=f(s_1,\dots,s_{n-1})\in \m^{[p^{n-1}]}
\]
by the definition of $\bs(f)$.

\begin{claim}\label{claim:fkdelta}
For $1\le k,t\le p-1$, if $H_{k,t}\in \m^{[p^n]}$ and $s_{n-1}\ge 1$, then
\[
H_{k-1,t}J(f)\subseteq \m^{[p^n]}.
\]
\end{claim}

\begin{proof}
Differentiating $H_{k,t}$, we obtain
\[
\frac{\partial H_{k,t}}{\partial x_i}
\equiv
kH_{k-1,t}f_i
+
tH_{k,t-1}f^{p-1}f_i
-
tL_{n-1}^pf^k\Delta(f)^{t-1}x_i^{p-1}f_i^p
\pmod p
\]
by Lemma~\ref{lem:delta-derivative}.
Since $H_{k,t}\in \m^{[p^n]}$, the derivative belongs to $\m^{[p^n]}$.
The second and third terms are contained in $\m^{[p^n]}$, because
\begin{align*}
H_{k,t-1}f^{p-1}
&=
(L_{n-1}f)^pf^{k-1}\Delta(f)^{t-1}\in \m^{[p^n]},\\
L_{n-1}^pf_i^p
&=(L_{n-1}f_i)^p\in \m^{[p^n]}
\end{align*}
by $(B_{n-1})$. Since $1\le k\le p-1$, we have
\[
H_{k-1,t}J(f)\subseteq \m^{[p^n]}.
\qedhere
\]
\end{proof}

Assume $1\le s_n\le p-1$. 
By definition of $s_n$, we have
\[
L_nf=f(s_1,\dots,s_n)\in \m^{[p^n]}.
\]

If $s_{n-1}=0$, then
\[
\frac{\partial(L_nf)}{\partial x_i}=\frac{\partial(L_{n-1}^pf^{p-s_n})}{\partial x_i}
\equiv (p-s_n)L_nf_i \pmod{p}.
\]
Since $L_nf\in \m^{[p^n]}$, the left-hand side belongs to $\m^{[p^n]}$. 
Since $1\le s_n\le p-1$, we have
\[
L_nf_i\in \m^{[p^n]}
\qquad (0\le i\le N).
\]
Thus
\[
L_n(f,J(f))\subseteq \m^{[p^n]}.
\]

Assume now that $1\le s_{n-1}$. Applying Claim~\ref{claim:fkdelta} to
\[
H_{p-s_n,s_{n-1}}
=
L_{n-1}^pf^{p-s_n}\Delta(f)^{s_{n-1}}
=
L_nf \in \m^{[p^n]},
\]
we obtain
\[
H_{p-s_n-1,s_{n-1}}J(f)=L_nJ(f)\subseteq \m^{[p^n]}.
\]
Together with $L_nf\in \m^{[p^n]}$, this proves
\[
L_n(f,J(f))\subseteq \m^{[p^n]}.
\]

\noindent
{\bf Proof of $(A_n)$.}
Suppose, for contradiction, that $s_n=p$. Then
\[
f(s_1,\dots,s_{n-1},p)=L_{n-1}^p\Delta(f)^{s_{n-1}}\in \m^{[p^n]}.
\]

If $s_{n-1}=0$, then $L_{n-1} \in \m^{[p^{n-1}]}$,
that is,
\[
L_{n-1}=f(s_1,\dots,s_{n-2},1)\in \m^{[p^{n-1}]},
\]
contrary to the maximality of $s_{n-1}=0$.

Thus we may assume $1\le s_{n-1}\le p-1$. By $(B_{n-1})$,
\[
L_{n-1}(f,J(f))\subseteq \m^{[p^{n-1}]}.
\]

We prove by descending induction on $t$, for $0\le t\le s_{n-1}$, that
\[
L_{n-1}^p\Delta(f)^t\in \m^{[p^n]}.
\]
The case $t=s_{n-1}$ is exactly the containment
\[
f(s_1,\dots,s_{n-1},p)=L_{n-1}^p\Delta(f)^{s_{n-1}}\in \m^{[p^n]}.
\]

It suffices to show the following claim.

\begin{claim}\label{cl:delta}
If $L_{n-1}^p\Delta(f)^t\in \m^{[p^n]}$ for $1\le t\le s_{n-1}$, then
\[
L_{n-1}^p\Delta(f)^{t-1}\in \m^{[p^n]}.
\]
\end{claim}

\begin{proof}
Differentiating $L_{n-1}^p\Delta(f)^t$, we obtain
\[
\frac{\partial\bigl(L_{n-1}^p\Delta(f)^t\bigr)}{\partial x_i}
\equiv
tL_{n-1}^pf^{p-1}\Delta(f)^{t-1}f_i
-
tL_{n-1}^p\Delta(f)^{t-1}x_i^{p-1}f_i^p
\pmod p
\]
by Lemma~\ref{lem:delta-derivative}. By assumption, the derivative belongs to $\m^{[p^n]}$.
The second term is contained in $\m^{[p^n]}$ by $(B_{n-1})$. Thus
\[
L_{n-1}^pf^{p-1}\Delta(f)^{t-1}f_i=H_{p-1,t-1}f_i\in \m^{[p^n]}
\qquad (0\le i\le N).
\]
Therefore,
\[
H_{p-1,t-1}J(f)\subseteq \m^{[p^n]}.
\]

Since $t-1 \leq s_{n-1}-1$, by \eqref{eq:degree-H}, we have
\[
\deg(H_{p-1,t-1}) <p^nQ+Q-(N+1)d.
\]
By Lemma~\ref{lem:jacobian-degree-membership}, we obtain
\[
H_{p-1,t-1}=L_{n-1}^pf^{p-1}\Delta(f)^{t-1}\in \m^{[p^n]}.
\]

If $t\ge 2$, then by Claim~\ref{claim:fkdelta},
\[
H_{p-2,t-1}J(f)=L_{n-1}^pf^{p-2}\Delta(f)^{t-1}J(f)\subseteq \m^{[p^n]}.
\]
By \eqref{eq:degree-H}, Lemma~\ref{lem:jacobian-degree-membership} yields
\[
H_{p-2,t-1}=L_{n-1}^pf^{p-2}\Delta(f)^{t-1}\in \m^{[p^n]}.
\]
Repeating this argument, we obtain
\[
L_{n-1}^p\Delta(f)^{t-1}\in \m^{[p^n]}.
\]

If $t=1$, then we already have
\[
L_{n-1}^pf^{p-1}\in \m^{[p^n]}.
\]
Differentiating $L_{n-1}^pf^{p-1}$, we get
\[
\frac{\partial (L_{n-1}^pf^{p-1})}{\partial x_i}
\equiv (p-1)L_{n-1}^pf^{p-2}f_i \pmod p,
\]
hence
\[
H_{p-2,0}J(f)=L_{n-1}^pf^{p-2}J(f)\subseteq \m^{[p^n]}.
\]
By \eqref{eq:degree-H} and  Lemma~\ref{lem:jacobian-degree-membership}, we have
\[
L_{n-1}^pf^{p-2}\in \m^{[p^n]}.
\]
Repeating this argument, we obtain
\[
L_{n-1}^p\in \m^{[p^n]}.
\]
\end{proof}

Thus, by descending induction on $t$, we obtain
\[
L_{n-1}\in \m^{[p^{n-1}]},
\]
that is,
\[
f(s_1,\dots,s_{n-2},s_{n-1}+1)\in \m^{[p^{n-1}]},
\]
contrary to the maximality of $s_{n-1}$.

This contradiction proves $(A_n)$. 
Hence, by induction, (1) holds for all $n$ satisfying $(N_n)$.

For (2), note that $(N_m)$ is equivalent to
\[
p^m(d-Q)+d(N-p)-Q<0.
\]
If $Q\ge d$, then the left-hand side is nonincreasing in $m$.
Hence $(N_2)$ implies $(N_m)$ for every $m\ge 2$.
The conclusion now follows from (1).

Finally, assume $Q=d$. Then $(N_2)$ becomes
\[
p>N-1=\dim X.
\]
Thus (3) follows from (2).
\end{proof}

\begin{corollary}\label{cor:CY-all-sn-bound}
Assume that the Jacobian ideal $J(\bar f)\subset \bar A$ is $\bar{\m}$-primary, $Q=d$ and $p>\dim X$. 
Then
\[
s_n\le \dim X
\qquad\text{for every }n\ge 1.
\]
In particular, $(A/f)^{\wedge p}$ is perfectoid pure and 
\[
\mathrm{ppt}((A/(f))^{\wedge p},p)\ge \frac{p-1-\dim X}{p-1},
\]
where $\mathrm{ppt}((A/(f))^{\wedge p},p)$ denotes the perfectoid purity threshold introduced in \cite{Yoshikawa25-cri}.
\end{corollary}

\begin{proof}
We note that $\dim X=N-1$.
By Theorem~\ref{thm:sn-le-p-minus-1}(2), the assumptions $Q=d$ and $p>\dim X$ imply
\[
s_n\le p-1
\qquad\text{for every }n\ge 1.
\]
We may assume $1\le s_n\le p-1$.
Set
\[
L_n:=f(s_1,\dots,s_{n-1},s_n+1).
\]
By $(B_n)$ in the proof of Theorem~\ref{thm:sn-le-p-minus-1}, we have
\[
L_n(f,J(f))\subseteq \m^{[p^n]}.
\]

If $s_n\ge N$, then
\[
\deg(L_n)=(p^n-s_n-1)d<(p^n-N)d=p^nQ+Q-(N+1)d
\]
by \eqref{eq:deg-L} and $Q=d$.
Hence Lemma~\ref{lem:jacobian-degree-membership} yields
\[
L_n=f(s_1,\dots,s_{n-1},s_n+1)\in \m^{[p^n]}.
\]
This contradicts the maximality of $s_n$.

Therefore $s_n\ge N$ is impossible, and hence
\[
s_n\le N-1=\dim X
\qquad\text{for every }n\ge 1.
\]

By \cite{Yoshikawa25-cri}*{Theorem~A and Proposition~6.2}, the ring $A/f$ is perfectoid pure and
\[
\mathrm{ppt}((A/(f))^{\wedge p},p)
=
\sum_{n=1}^{\infty}\frac{p-s_n-1}{p^n}
\ge
\sum_{n=1}^{\infty}\frac{p-1-\dim X}{p^n}
=
\frac{p-1-\dim X}{p-1},
\]
as desired.
\end{proof}

\begin{corollary}\label{cor:eventual-vanishing-fano}
Assume that $J(\bar f)\subset \bar A$ is $\bar{\m}$-primary.
If $Q>d$ and $(N_2)$ holds, then
\[
s_n=0
\qquad\text{for all }n\gg0.
\]
In particular, the localization $(A/f)_\m$ is perfectoid BCM-regular.
\end{corollary}

\begin{proof}
By Theorem~\ref{thm:sn-le-p-minus-1}(2), we have
\[
s_n\le p-1
\qquad\text{for every }n\ge 1.
\]
Since $Q>d$, for all sufficiently large $n$,
\[
d(p^n-2)<p^nQ+Q-(N+1)d.
\]
Fix such an $n$. If $s_n\ge 1$, then
\[
L_n(f,J(f))\subseteq \m^{[p^n]},
\qquad
L_n:=f(s_1,\dots,s_{n-1},s_n+1).
\]
by $(B_n)$ in the proof
of Theorem~\ref{thm:sn-le-p-minus-1}.
Moreover,
\[
\deg(L_n)=d(p^n-s_n-1)\le d(p^n-2)<p^nQ+Q-(N+1)d
\]
by \eqref{eq:deg-L}.
Hence Lemma~\ref{lem:jacobian-degree-membership} yields $L_n\in \m^{[p^n]}$.
This contradicts the maximality of $s_n$. Therefore $s_n=0$ for all sufficiently large
$n$.

The final assertion follows from \cite{Yoshikawa25-cri}*{Theorem~6.4} since $\var{A/f}$ has isolated singularity.
\end{proof}

\begin{theorem}\label{general-fano}
Assume that $(\bar f,J(\bar f))\subset \bar A$ is $\bar{\m}$-primary.
If $Q>d$ and $s_n\le p-1$ for every $n\ge 1$,
then
\[
s_n=0
\qquad\text{for all }n\gg0.
\]
In particular, the localization $(A/f)_\m$ is perfectoid BCM-regular.
\end{theorem}

\begin{proof}
Since $(\bar f,J(\bar f))\subset \bar A$ is $\bar{\m}$-primary, there exists a positive integer $k$ such that
\[
\bar A_{\ge k}\subseteq (\bar f,J(\bar f)).
\]
By Lemma~\ref{lem:weighted-colon-estimate} and the proof of Lemma~\ref{lem:jacobian-degree-membership}, if a homogeneous element $H\in A$ satisfies
\[
H(f,J(f))\subseteq \m^{[p^e]}
\qquad\text{and}\qquad
\deg(H)<p^eQ-Q-k+1,
\]
then
\[
H\in \m^{[p^e]}.
\]

Since $Q>d$, we have
\[
d(p^n-2)<p^nQ-Q-k+1
\]
for all sufficiently large $n$.
Fix such an $n$. We prove that $s_n=0$.

Suppose, for contradiction, that $s_n\ge 1$. Since $s_n\le p-1$ by assumption, the proof of Theorem~\ref{thm:sn-le-p-minus-1} yields
\[
f(s_1,\ldots,s_{n-1},s_n+1)\,(f,J(f))\subseteq \m^{[p^n]}.
\]
Moreover,
\[
\deg\bigl(f(s_1,\ldots,s_{n-1},s_n+1)\bigr)
=
d(p^n-s_n-1)
\le d(p^n-2)
<
p^nQ-Q-k+1,
\]
where the first equality follows from \eqref{eq:deg-L}. Hence, by the above criterion, we obtain
\[
f(s_1,\ldots,s_{n-1},s_n+1)\in \m^{[p^n]}.
\]
This contradicts the definition of $s_n$.
Therefore $s_n=0$ for all sufficiently large $n$.

The final assertion follows from \cite{Yoshikawa25-cri}*{Theorem~6.4} since $\var{A/f}$ has isolated singularity.
\end{proof}

\section{Applications to Calabi–Yau and Fano hypersurfaces}
In this section, we apply the numerical results of Section~2 to Calabi--Yau and Fano hypersurfaces. 
We first discuss hypersurface lifts, and then prove the main applications to perfectoid splitting and global $+$-regularity.

\subsection{Hypersurface lifts}
We first prove a lifting result for smooth weighted hypersurfaces, which shows that under the assumptions needed later, every unramified lift is again a weighted hypersurface.

\begin{proposition}\label{prop:hypersurface-lift}
Let $C$ be a complete unramified discrete valuation ring with maximal ideal $(p)$ and $C/pC=k$, and let
\[
X=\Proj\bigl(k[x_0,\dots,x_N]/(f)\bigr)\subset \mathbb P_k(q_0,\dots,q_N)
\]
be a smooth weighted hypersurface. Assume that one of the following holds:
\begin{enumerate}
    \item $X$ is Fano;
    \item $X$ is Calabi--Yau and $\dim X\neq 2$.
\end{enumerate}
Then every unramified lift $\widetilde X$ of $X$ over $C$ is isomorphic to a weighted
hypersurface
\[
\widetilde X \cong \Proj\bigl(C[x_0,\dots,x_N]/(\widetilde f)\bigr)
\subset \mathbb P_C(q_0,\dots,q_N)
\]
for some homogeneous lift $\widetilde f$ of $f$.
\end{proposition}

\begin{proof}
Set
\[
P:=\mathbb P_k(q_0,\dots,q_N),
\qquad
Q:=q_0+\cdots+q_N.
\]
Since $X\subset P$ is a hypersurface of degree $d$, we have the exact sequence
\[
0\to \mathcal O_{P}(-d)\to \mathcal O_{P}\to \mathcal O_X\to 0.
\]
First, we prove $H^2(X,\cO_X)=0$.
By \cite{Shin-weighted-Brauer}*{Lemma~2.9(3)}, one has
\[
H^i(\mathbb P,\mathcal O_{\mathbb P}(m))=0
\qquad\text{for every }m\in \mathbb Z\text{ and every }0<i<N.
\]
Taking cohomology, if $\dim X \neq 2$, then
\[
H^2(X,\mathcal O_X)=0.
\]
On the other hand, if $\dim X=2$ and $d<Q$, then \cite{Shin-weighted-Brauer}*{Lemma~2.9(2)} yields that
$H^3(\mathbb P,\mathcal O_{\mathbb P}(-d))$ has a basis consisting of monomials
\[
x_0^{e_0}\cdots x_3^{e_3}
\qquad (e_i<0)
\]
of weighted degree $-d$. But every such monomial has degree at most
\[
-(q_0+\cdots+q_3)=-Q,
\]
whereas $-d>-Q$ because $d<Q$. Therefore
\[
H^3(\mathbb P,\mathcal O_{\mathbb P}(-d))=0.
\]
Thus again
\[
H^2(X,\mathcal O_X)=0.
\]

Since $X$ is smooth, the reflexive sheaf $\mathcal O_{\mathbb P}(1)|_X$ is a line bundle,
which we denote by $\mathcal O_X(1)$. 
The obstruction to lifting $\mathcal O_X(1)$ to
$\widetilde X$ lies in $H^2(X,\mathcal O_X)$, so $\mathcal O_X(1)$ lifts to a line bundle
$\widetilde L$ on $\widetilde X$.

Next we prove that the natural map
\[
H^0(\wt{X},\wt{L}^{\otimes m}) \to H^0(X,\cO_X(m))
\]
is surjective for $m\ge 0$.
If $\dim X \geq 2$, from
\[
0\to \mathcal O_{\mathbb P}(m-d)\to \mathcal O_{\mathbb P}(m)\to \mathcal O_X(m)\to 0
\]
we obtain $H^1(X,\cO_X(m))=0$ by \cite{Shin-weighted-Brauer}*{Lemma~2.9(3)}.
Next, if $\dim X=1$, then $H^1(X,\cO_X(m))=0$ for $m \geq 1$ by Kodaira-type vanishing.
Furthermore, since $H^0(\wt{X},\cO_{\wt{X}})$ is a $C$-algebra and $H^0(X,\cO_X)=k$, we obtain the desired surjection.

Now set
\[
R:=\bigoplus_{m\ge 0}H^0(\widetilde X,\widetilde L^{\otimes m}).
\]
Since $\widetilde L$ is ample, we have $\widetilde X\cong \Proj(R)$.
By cohomology and base change, each graded piece of $R$ is a finite free $C$-module and
\[
R/pR \cong \bigoplus_{m\ge 0}H^0(X,\mathcal O_X(m))
\cong k[x_0,\dots,x_N]/(f).
\]

Taking lifts, we can define a graded $C$-algebra homomorphism
\[
\varphi\colon C[x_0,\dots,x_N]\to R.
\]
Modulo $p$, this becomes the natural surjection
\[
k[x_0,\dots,x_N]\twoheadrightarrow k[x_0,\dots,x_N]/(f),
\]
so $\varphi$ is surjective by graded Nakayama.

Let $I:=\ker(\varphi)$. Since both $C[x_0,\dots,x_N]$ and $R$ are flat over $C$, the ideal
$I$ is flat over $C$. Reducing modulo $p$, we get $I/pI=(f)$.
Choose a homogeneous lift $\widetilde f\in I$ of $f$. 
Then $I=(\widetilde f)+pI$,
hence graded Nakayama implies
\[
I=(\widetilde f).
\]
Therefore
\[
R\cong C[x_0,\dots,x_N]/(\widetilde f),
\]
and so
\[
\widetilde X\cong \Proj(R)\cong
\Proj\bigl(C[x_0,\dots,x_N]/(\widetilde f)\bigr)
\subset \P_C(q_0,\dots,q_N).
\qedhere
\]
\end{proof}

\subsection{Calabi--Yau hypersurfaces}
We now apply the criterion of Section~2 to Calabi--Yau hypersurfaces. 
This yields perfectoid splitting for smooth unramified lifts in residue characteristic larger than the dimension, together with a corresponding local statement for families.

\begin{theorem}\label{CY}
Let $C$ be a complete discrete valuation ring with maximal ideal $(p)$ and $k:=C/(p)$.
Let $\wt{X}$ be a Calabi--Yau hypersurface in a well-formed weighted projective space over $C$.
We assume $X:=\wt{X} \times_{\Spec C} \Spec k$ is quasi-smooth in $\P_k(q_0,\ldots,q_N)$.
If $p$ is larger than the relative dimension of $\wt{X}$ and $p \nmid (q_0+\cdots +q_N)$, then $\wt{X}$ is perfectoid split.  
In particular, every unramified lift of a smooth Calabi--Yau hypersurface $X$ is perfectoid split if $p>\dim X$, $\dim X \neq 2$, and $p \nmid (q_0+\cdots +q_N)$ .
\end{theorem}

\begin{proof}
It follows from Corollary~\ref{cor:CY-all-sn-bound}, \cite{IY26}*{Proposition~5.3}, and Proposition~\ref{prop:hypersurface-lift}.
We note that since $p \nmid (q_0+\cdots +q_N)$, the quasi-smoothness condition implies 
\[
(\frac{\partial f}{\partial x_0},\ldots ,\frac{\partial f}{\partial x_N})
\]
is $(x_0,\ldots,x_N)$-primary, where $f$ is the defining equation of $X$.
\end{proof}

\begin{corollary}\label{weak-ordinary}
Let $S$ be a scheme which is smooth and dominant over $\Spec \mathbb{Z}$, and let $Y_S \to S$ be a smooth proper morphism whose fibers are Calabi--Yau varieties of dimension $n$.
We assume $Y_S$ is a hypersurface in a projective space over $S$.
Let $\fp \in S$ such that $\kappa(\fp)$ is a field of positive characteristic.
Then the localization
\[
Y_{S,\mathfrak p}:=Y_S\times_S \Spec \mathcal{O}_{S,\mathfrak p}
\]
is perfectoid split if $\mathrm{char}(\kappa(\fp))>n$ and $\mathrm{char}(\kappa(\fp)) \neq n+2$.
\end{corollary}

\begin{proof}
We take a point $\fp \in S$ such that $\kappa(\fp)$ is of characteristic $p >n$ and $p \neq n+2$.
We set $R:=\cO_{S,\fp}$. Then it is a smooth and local $\Z$-algebra.
The maximal ideal of $R$ is denoted by $\m$.
Since $R$ is smooth over $\Z$ and has residue characteristic $p$, we can take a regular system of parameters $p,y_1,\ldots,y_r$.
By assumption, $Y_{S,\mathfrak p}$ is a hypersurface in $\P^{n+1}_R$ of degree $n+2$.
The defining equation is denoted by 
\[
f \in R[x_0,\ldots,x_{n+1}].
\]
Then $f$ is homogeneous of degree $n+2$.
Set $C:=R/(y_1,\ldots,y_r)$.
Since $C$ is a complete discrete valuation ring with maximal ideal $(p)$, by Theorem~\ref{thm:sn-le-p-minus-1} (3), the ring
\[
C[x_0,\ldots,x_{n+1}]/(f)
\]
is perfectoid pure since
\[
\Proj\bigl((C[x_0,\ldots,x_{n+1}]/(f))\otimes_C k\bigr)=Y_{\mathfrak p}
\]
is smooth.
By \cite{p-pure}*{Theorem~6.6} and \cite{IY26}*{Proposition~5.3}, the scheme $Y_{S,\mathfrak p}$ is perfectoid split, as desired.
\end{proof}

\begin{theorem}\label{p-purity-K3}
Let $X$ be a K3 surface over an algebraically closed field of characteristic $p>0$.
Assume that $X$ is a hypersurface in a projective space.
\begin{enumerate}
    \item If $p>2$ or the Artin--Mazur height of $X$ is finite, then every unramified lift as a hypersurface of $X$ is perfectoid split.
    \item If $p=2$ and $X$ is supersingular, then $X$ has unramified lifts $\wt{X}_1, \wt{X}_2$ such that $\wt{X}_1$ is perfectoid pure and $\wt{X}_2$ is not perfectoid pure.
\end{enumerate}
\end{theorem}

\begin{proof}
If $p>2$, then the result follows from Theorem~\ref{CY}.
If the Artin--Mazur height of $X$ is finite, then $X$ is quasi-$F$-split by \cite{Yobuko}*{Theorem~4.5}.
Thus, the result follows from \cite{Yoshikawa25}*{Theorem~A}.

Next, we assume $p=2$ and $X$ is supersingular.
Then the result follows from \cite{TY26-K3}.
\end{proof}

\subsection{Fano hypersurfaces}
Finally, we consider Fano hypersurfaces. Using the eventual vanishing result for the splitting-order sequence, we prove global $+$-regularity for unramified lifts and derive explicit consequences in threefold cases.

\begin{theorem}\label{thm:Fano}
Let $C$ be a complete discrete valuation ring with maximal ideal $(p)$ and $k:=C/(p)$.
Let $\wt{X}$ be a Fano hypersurface of degree $d$ in a well-formed weighted projective space
$\mathbb{P}_C(q_0,\ldots,q_N)$, and set
\[
Q:=q_0+\cdots+q_N.
\]
We assume $X:=\wt{X} \times_{\Spec C} \Spec k$ is quasi-smooth in $\P_k(q_0,\ldots,q_N)$.
If $p \nmid d$ and
\[
(Q-d)p^2+dp+Q-Nd>0,
\]
then $\wt{X}$ is globally $+$-regular.
In particular, if $p\ge \dim X$ and $p \nmid d$, then $\wt{X}$ is globally $+$-regular.
\end{theorem}

\begin{proof}
Let $A=C[x_0,\ldots,x_N]/(f)$ be a coordinate ring of $\wt{X}$.
Since $\var{A}$ has an isolated singularity, by Corollary~\ref{cor:eventual-vanishing-fano}, the localization $A_\m$ is perfectoid BCM-regular, where $\m:=(p,x_0,\ldots,x_N)$, and in particular, $A_\m$ is $+$-regular.
By \cite{Onuki}*{Lemma~2.2}, the lift $\wt{X}$ is globally $+$-regular.

For the last assertion, it suffices to show that if $p\ge \dim X=N-1$, then the condition $(N_2)$ is satisfied.

Now $p\ge N-1$ implies
\[
dp-Nd=d(p-N)\ge -d.
\]
Hence
\begin{align*}
(Q-d)p^2+dp+Q-Nd
&\ge (Q-d)p^2+Q-d.
\end{align*}
Since $X$ is Fano, we have $Q>d$, so the right-hand side is strictly positive. Therefore
$(N_2)$ holds, and the result follows.
\end{proof}

\begin{corollary}\label{Fano-smooth}
Let $C$ be a complete discrete valuation ring with maximal ideal $(p)$ and $k:=C/(p)$.
Let $X$ be a smooth Fano hypersurface in a  weighted projective space
$\mathbb{P}_C(q_0,\ldots,q_N)$, and set
\[
Q:=q_0+\cdots+q_N.
\]
If $p \nmid d$ and
\[
(Q-d)p^2+dp+Q-Nd>0,
\]
then every unramified lift of $X$ is globally $+$-regular.
In particular, if $p\ge \dim X$ and $p \nmid d$, then every unramified lift of $X$ is globally $+$-regular.
\end{corollary}

\begin{proof}
It follows from Theorem~\ref{thm:Fano} and Proposition~\ref{prop:hypersurface-lift}.
\end{proof}

\begin{proposition}\label{prop:smooth-fano-3fold-weights}
Let \(k\) be a perfect field of characteristic \(p>0\), and let
\[
X=X_d\subset \P(q_0,\dots,q_4)
\]
be a smooth Fano threefold hypersurface over \(k\).
Then the ambient weight system \((q_0,\dots,q_4)\), up to reordering, is one of the following:
\[
(1,1,1,1,1),\qquad
(1,1,1,1,2),\qquad
(1,1,1,1,3),\qquad
(1,1,1,2,3).
\]
Moreover, the corresponding possibilities for the degree are:
\[
\begin{array}{c|c}
(q_0,\dots,q_4) & d\\
\hline
(1,1,1,1,1) & 1,2,3,4,\\
(1,1,1,1,2) & 4,\\
(1,1,1,1,3) & 6,\\
(1,1,1,2,3) & 6.
\end{array}
\]
\end{proposition}

\begin{proof}
We may assume $k$ is an algebraically closed field.
By the classification of smooth Fano threefolds in positive characteristic, a smooth Fano
threefold \(X\) with \(\rho(X)=1\) is one of the varieties listed in Table~14 of
\cite{Tanaka-pos-IV}. 
More precisely, Table~14 is obtained from
\cite{Tanaka-pos-I}*{Theorem~2.18 and Theorem~2.23},
\cite{Tanaka-pos-II}*{Theorem~1.1 and Proposition~2.8},
and the index-one non-very-ample case treated in \cite{Tanaka-pos-I}*{Theorem~1.1 and Proposition~6.8}.

Among the varieties appearing in Table~14, the ones that are smooth hypersurfaces in a
weighted projective \(4\)-space are exactly the following:
\[
\P^3,\qquad
Q^3\subset \P^4,\qquad
X_3\subset \P^4,\qquad
X_4\subset \P^4,
\]
\[
X_4\subset \P(1,1,1,1,2),\qquad
X_6\subset \P(1,1,1,1,3),\qquad
X_6\subset \P(1,1,1,2,3).
\]
Therefore the only possible ambient weight systems are
\[
(1,1,1,1,1),\ (1,1,1,1,2),\ (1,1,1,1,3),\ (1,1,1,2,3).
\]
The corresponding degrees are read off directly from the same list:
\(\P^3\), \(Q^3\), \(X_3\subset \P^4\), and \(X_4\subset \P^4\) give
\(d=1,2,3,4\) for the weight system \((1,1,1,1,1)\), while the weighted cases are
\[
X_4\subset \P(1,1,1,1,2),\qquad
X_6\subset \P(1,1,1,1,3),\qquad
X_6\subset \P(1,1,1,2,3).
\]
This proves the claim.
\end{proof}

\begin{corollary}\label{cor:C2-smooth-fano-3fold}
Let \(k\) be a perfect field of characteristic \(p>3\), and let $X$ be a smooth Fano threefold hypersurface in a weighted projective space over \(k\).
Then every $W(k)$-lift of $X$ is globally $+$-regular.
\end{corollary}

\begin{proof}
We assume $X$ is a hypersurface in
\[
\P(q_0,\ldots,q_4)
\]
with defining equation of degree $d$.
We set $Q:=q_0+\cdots+q_4$. Then, by Proposition~\ref{prop:smooth-fano-3fold-weights}, the possible pairs $(Q,d)$ are
\[
(Q,d)=(5,1),\ldots,(5,4),(6,4),(8,6),(7,6).
\]

Since here \(N=4\) and $p\nmid d$, the condition in Theorem~\ref{thm:Fano} is satisfied.
Thus, the result follows from Theorem~\ref{thm:Fano}. 
\end{proof}

\begin{theorem}\label{thm:general-Fano}
Let $C$ be a complete discrete valuation ring with maximal ideal $(p)$ and $k:=C/(p)$.
Let
\[
\wt{X}=(f=0) \subseteq \mathbb{P}_C(q_0,\ldots,q_N)
\]
be a Fano hypersurface of degree $d$ in a well-formed weighted projective space.
We assume $X:=\wt{X} \times_{\Spec C} \Spec k$ is quasi-smooth in $\P_k(q_0,\ldots,q_N)$.
If the splitting-order sequence $(s_i)_{i \geq 0}$ of $f$ satisfies $s_n \leq p-1$ for every $n \geq 1$, then $\wt{X}$ is globally $+$-regular.
\end{theorem}

\begin{proof}
It follows from Theorem~\ref{general-fano} and \cite{Onuki}*{Lemma~2.2}.
\end{proof}

\bibliographystyle{skalpha}
\bibliography{bibliography.bib}
\end{document}